\documentclass[reqno,12pt]{amsart}
\usepackage{latexsym,amssymb,amsmath,epsfig}
\usepackage{amsthm}
\usepackage{enumerate}
\usepackage{mathrsfs}
\usepackage{graphicx}
\usepackage{pstricks-add}
\usepackage[latin1]{inputenc} 
\newtheorem{theorem}{Theorem}[section]
\newtheorem{lemma}{Lemma}[section]
\newtheorem{corollary}{Corollary}[section]

\newtheorem{remark}{Remark}[section]

\newtheorem{example}{Example}[section]

\begin{document}
\title{(Di)graph decompositions and magic type labelings: a dual relation}
\author{S. C. L\'opez}
\address{%
Departament de Matem\`{a}tiques\\
Universitat Polit\`{e}cnica de Catalunya.\\
C/Esteve Terrades 5\\
08860 Castelldefels, Spain}
\email{susana.clara.lopez@upc.edu}

\author{F. A. Muntaner-Batle}
\address{Graph Theory and Applications Research Group \\
 School of Electrical Engineering and Computer Science\\
Faculty of Engineering and Built Environment\\
The University of Newcastle\\
NSW 2308
Australia}
\email{famb1es@yahoo.es}

\author{M. Prabu}
\address{British University Vietnam\\
Hanoi, Vietnam}
\email{mprabu201@gmail.com}

\maketitle

\begin{abstract}
A graph $G$ is called edge-magic if there is a bijective function $f$ from the set of vertices and edges to the set $\{1,2,\ldots,|V(G)|+|E(G)|\}$ such that the sum $f(x)+f(xy)+f(y)$ for any $xy$ in $E(G)$ is constant. Such a function is called an edge-magic labeling of G and the constant is called the valence of $f$. An edge-magic labeling with the extra property that $f(V(G))= \{1,2,\ldots,|V(G)|\}$ is called super edge-magic. In this paper, we establish a relationship between the valences of (super) edge-magic labelings of certain types of bipartite graphs and the existence of a particular type of decompositions of such graphs.
\end{abstract}

\begin{quotation}

\end{quotation}

\section{Introduction}
For the terminology and notation not introduced in this paper we refer the reader to either one of the following sources \cite{BaMi,CharLes,G,SlMb,Wa}.
By a $(p,q)$-graph we mean a graph of order $p$ and size $q$.
Let $m\leq n$ be integers, to denote the set $\{m,m+1,\ldots,n\}$ we use $[m,n]$. Kotzig and Rosa  introduced in \cite{KotRos70} the concepts of edge-magic graphs and edge-magic labelings as follows: Let $G$ be a $(p,q)$-graph. Then $G$ is called {\it edge-magic} if there is a bijective function $f:V(G)\cup E(G)\rightarrow [1,p+q]$ such that the sum $f(x)+f(xy)+f(y)=k$ for any $xy\in E(G)$. Such a function is called an {\it edge-magic labeling} of $G$ and $k$ is called the {\it valence} \cite{KotRos70} or the {\it magic sum} \cite{Wa} of the labeling $f$. We write $\hbox{val}(f)$ to denote the valence of $f$.

Inspirated by the notion of edge-magic labelings, Enomoto et al. introduced in \cite{E} the concepts of super edge-magic graphs and super edge-magic labelings as follows: Let $f:V(G)\cup E(G) \rightarrow [1,p+q]$ be an edge-magic labeling of a $(p,q)$-graph G with the extra property that $f(V(G))=[1,p]$. Then G is called { \it super edge-magic} and $f$ is a {\it super edge-magic labeling} of $G$. Notice that although the definitions of (super) edge-magic graphs and labelings were originally provided for simple graphs (that is,  graphs with no loops nor multiple edges), along this paper, we understand these definitions for any graph. Therefore, unless otherwise specified, the graphs considered in this paper are not necessarily simple. Figueroa-Centeno et al. provided in \cite{F2}, the following useful characterization of super edge-magic simple graphs, that works in exactly the same way for non necessarily simple graphs.

\begin{lemma}\label{super_consecutive} \cite{F2}
Let $G$ be a $(p,q)$-graph. Then $G$ is super edge-magic if and only if  there is a
bijective function $g:V(G)\longrightarrow [1,p]$ such
that the set $S=\{g(u)+g(v):uv\in E(G)\}$ is a set of $q$
consecutive integers.
In this case, $g$ can be extended to a super edge-magic labeling $f$ with valence $p+q+\min S$.
\end{lemma}
Unless otherwise specified, whenever we refer to a function as a super edge-magic labeling we will assume that it is a function $f$ as in  Lemma \ref{super_consecutive}. Before moving on, it is worthwhile mentioning that Acharya and Hegde had already defined in 1991 \cite{AH} the concept of strongly indexable graphs. This concept turns out to be equivalent to the concept of super edge-magic graphs. However in this paper we will use the names super edge-magic graphs and super edge-magic labelings. In \cite{F1} Figueroa et al., introduced the concept of super edge-magic digraph as follows: a digraph $D=(V,E)$ is super edge-magic if its underlying graph is super edge-magic. In general, we say that a digraph D admits a labeling $f$ if its underlying graph admits the labeling $f$.
It was also in \cite{F1} that the following product was introduced: let $D$ be a digraph and let $\Gamma$ be a family of digraphs with the same set $V$ of vertices. Assume that $h: E(D) \to \Gamma$ is any function that assigns elements of $\Gamma$ to the arcs of $D$. Then the digraph $D \otimes _{h} \Gamma $ is defined by (i) $V(D \otimes _{h} \Gamma)= V(D) \times V$ and (ii) $((a,i),(b,j)) \in E(D \otimes _{h} \Gamma) \Leftrightarrow (a,b) \in E(D)$ and $(i,j) \in E(h(a,b))$. Note that when $h$ is constant, $D \otimes _{h} \Gamma$ is the Kronecker product.
Many relations among labelings have been established using the $\otimes_h$-product and some particular families of graphs, namely $\mathcal{S}_p$ and $\mathcal{S}_p^k$ (see for instance, \cite{ILMR,LopMunRiu1,LopMunRiu6,LopMunRiu7}).
The family $\mathcal{S}_p$ contains all super edge-magic $1$-regular labeled digraphs of order $p$ where each vertex takes the name of the label that has been assigned to it. A super edge-magic digraph $F$ is in $\mathcal{S}_p^k$ if $|V(F)|= |E(F)|=p$ and the minimum sum of the labels of adjacent vertices is equal to $k$ (see Lemma \ref{super_consecutive}). Notice that, since each $1$-regular digraph has minimum induced sum equal to $(p+3)/2$, $ \mathcal{S}_p \subset \mathcal{S}_p^{(p+3)/2}$. The following result was introduced in \cite{LopMunRiu6}, generalizing a previous result found in \cite{F1} :

\begin{theorem} \label{spk} \cite{LopMunRiu6}
Let $D$ be a (super) edge-magic digraph and let $h: E(D) \to \mathcal{S}_p^k$ be any function. Then und($D\otimes _{h} \mathcal{S}_p^k$) is (super) edge-magic.
\end{theorem}

\begin{remark}\label{remarkspk}
The key point in the proof of Theorem \ref{spk} is to rename the vertices of $D$ and each element of $\mathcal{S}_p^k$ after the labels of their corresponding (super) edge-magic labeling $f$ and their super edge-magic labelings respectively. Then the labels of the product are defined as follows: (i) the vertex $(a,i) \in V(D\otimes _{h} \mathcal{S}_p^k)$ receives the label: $p(a-1)+i$ and (ii) the arc $((a,i),(b,j)) \in E(D\otimes _{h} \mathcal{S}_p^k)$ receives the label: $p(e-1)+(k+p)-(i+j)$, where $e$ is the label of $(a,b)$ in D. Thus, for each arc $((a,i),(b,j)) \in E(D\otimes _{h} \mathcal{S}_p^k)$, coming from an arc $ e = (a,b) \in E(D)$ and an arc $ (i,j) \in E(h(a,b))$, the sum of labels is constant and equal to $p(a+b+e-3)+(k+p)$. That is, $p( \hbox{val}(f)-3)+k+p$. Thus, the next result is obtained.
\end{remark}

\begin{lemma}\label{valenceinducedproduct} \cite{LopMunRiu6}
Let $\widehat{f}$ be the (super) edge-magic labeling of the graph $D \otimes_h\mathcal{S}_p^k$ induced by a (super) edge-magic labeling $f$ of $D$ (see Remark \ref{remarkspk}). Then the valence of $\widehat{f}$ is given by the formula
\begin{eqnarray}
\hbox{val}(\widehat{f}) &=& p(\hbox{val}(f)-3) + k + p.
\end{eqnarray}
\end{lemma}

All the results in the literature involving the $\otimes_h$-product had super edge-magic labeled digraphs in the second factor of the product. However, in \cite{LMP2} it was shown that other labeled (di)graphs can be used in order to enlarge the results obtained, showing that the $\otimes_h$-product is a very powerful tool.  Next, we introduce the family $\mathcal{T}^q_\sigma$ of edge-magic labeled digraphs. An edge-magic labeled digraph $F$ is in $\mathcal{T}^q_\sigma$ if $V(F)=V$, $|E(F)|=q$ and the magic sum of the edge-magic labeling is equal to $\sigma$.

\begin{theorem}\cite{LMP2}
\label{producte_super_k}
Let $D\in \mathcal{S}_n^k$ and let $h$ be any function $h:E(D)\rightarrow \mathcal{T}^q_\sigma$. Then $D\otimes_h \mathcal{T}^q_\sigma$ admits an edge-magic labeling with valence  $(p+q)(k+n-3)+\sigma$, where $p=|V|, \ |E(F)|=q$ and $F \in \mathcal{T}^q_\sigma$.
\end{theorem}

\begin{remark}
Let $p=|V|$. The keypoint in the proof of Theorem \ref{producte_super_k} is to identify the vertices of $D$ and each element of $\mathcal{T}^q_\sigma$ after the labels of their corresponding super edge-magic labeling and edge-magic labeling, respectively. Then the labels of $D\otimes_h \mathcal{T}^q_\sigma$ are defined as follows: (i) if $(i,a)\in V(D\otimes_h \mathcal{T}^q_\sigma)$ we assign to the vertex the label: $(p+q)(i-1)+a$ and (ii) if $((i,a),(j,b))\in E (D\otimes_h \mathcal{T}^q_\sigma)$ we assign to the arc the label: $(p+q)(k+n-(i+j)-1)+(\sigma-(a+b)).$ Notice that, since $D\in \mathcal{S}_n^k$ is labeled with a super edge-magic labeling with minimum sum of the adjacent vertices equal to  $k$, we have $\{(k+n)-(i+j): \ (i,j)\in E(D )\}=[1, n].$ Moreover, since each element $F\in \mathcal{T}^q_\sigma$, it follows that $\{(\sigma-(a+b): \ (a,b)\in E(F )\}=[1, p+q]\setminus V .$ Thus, the set of labels in $D\otimes_h \mathcal{T}^q_\sigma$ covers all elements in $[1, n(p+q)]$. Moreover, for each arc $((i,a)(j,b))\in E (D\otimes_h \mathcal{T}^q_\sigma)$ the sum of the labels is constant and is equal to: $(p+q)(k+n-3)+\sigma.$
\end{remark}

In \cite{LopMunRiu5} L\'opez et al. introduced the following definitions.
Let $G=(V,E)$ be a $(p,q)$-graph. Then the set $S_{G}$ is defined as $S_{G}= \{ 1/q( \Sigma_{u \in V}  deg(u)g(u)+ \Sigma_{i=p+1}^{p+q} i ):$ the function $g:V \rightarrow [1,p]$ is bijective\}. If $\lceil\min S_G\rceil\le  \lfloor\max S_G\rfloor$ then the {\it super edge-magic interval} of $G$, denoted by $I_G$, is defined to be the set $I_G=\left[\lceil\min S_G\rceil, \lfloor\max S_G\rfloor\right]$
and the {\it super edge-magic set} of $G$, denoted by $\sigma_G$, is the set formed by all integers $k\in I_G$ such that $k$ is the valence of some super edge-magic labeling of $G$. A graph $G$ is called {\it perfect super edge-magic} if $\sigma_G=I_G$.
In order to conduct our study in this paper, the following lemma will be of great help.
\begin{lemma}\cite{LMP2} \label{k1nsem}
The graph formed by a star $K_{1,n}$ and a loop attached to its central vertex, denoted by $K_{1,n}^{l}$, is perfect super edge-magic for all positive integers $n$. Furthermore, $|I_{K_{1,n}^{l}}|=|\sigma_{K_{1,n}^{l}}|=n+1$.
\end{lemma}

In \cite{PEM_LMR} the same authors generalized the previous definitions to edge-magic graphs and labelings as follows:
Let $G=(V,E)$ be a $(p,q)$-graph, and denote by $T_G$ the set $$\left\{\frac{\sum_{u\in V}\mbox{deg}(u)g(u)+\sum_{e\in E}g(e)}q:\ g:V\cup E \rightarrow [1,p+q] \ \mbox{ is a bijective function}\right\}.$$
If $\lceil\min T_G\rceil\le  \lfloor\max T_G\rfloor$ then the {\it magic interval} of $G$, denoted by $J_G$, is defined to be the set
$J_G=\left[\lceil\min T_G\rceil, \lfloor\max T_G\rfloor\right]$ and the {\it magic set} of $G$, denoted by $\tau_G$, is the set
$\tau_G=\{n\in J_G:\ n \ \mbox{is the valence of some edge-magic labeling of}\ G\}.$ It is clear that $\tau_G\subseteq J_G$. A graph $G$ is called {\it perfect edge-magic } if $\tau_G=J_G$.  In the next lemma, we provide a well known result that gives a lower bound and an upper bound for edge-magic valences. We add the proof as a matter of completeness. Recall that the complementary labeling of an edge-magic labeling $f$ is the labeling $\overline{f}(x)=p+q+1-f(x)$, for all  $x \in V(G) \cup E(G)$, and that val$(\overline{f})=3(p+q+1)-\hbox{val}(f)$.

\begin{lemma}\label{maxminvalence}
Let $G$ be a $(p,q)$-graph with an edge-magic labeling $f$. Then $p+q+3 \leq \hbox{val}(f) \leq 2(p+q).$
\end{lemma}

\begin{proof}
Let $f:V(G) \cup E(G) \rightarrow [1,p+q]$ be an edge-magic labeling of $G$. The two lowest possible integers in $[1,p+q-1]$ that can be added to $p+q$ are $1$ and $2$. Thus, $\hbox{val}(f) \geq p+q+3.$ By using the complementary labeling, the maximum possible valence has the form $3(p+q+1)-\hbox{val}(g)$ where $\hbox{val}(g)$ is the minimum possible valence. Thus, $\hbox{val}(f) \leq 3(p+q+1)-\hbox{val}(g)\leq 2(p+q).$
\end{proof}

The study of the (super) edge-magic properties of the graph $C_m\odot \overline{K}_n$ as a particular subfamily of $S_n^k$ has been of interest recently. See for instance \cite{LMP1,LopMunRiu5,PEM_LMR}. Due to this, many things are known on the (super) edge-magic properties of the graphs $C_{p^k} \otimes \overline{K}_n$ \cite{PEM_LMR} and $C_{pq} \otimes \overline{K}_n$ \cite{LMP1}, where $p$ and $q$ are coprime. However, many other things remain a mystery, and we believe that it is worth the while to work in this direction. In fact, a big hole in the literature, appears when considering graphs of the form $C_m\odot \overline{K}_n$ for $m$ even. In this paper, we will devote Section \ref{section_morevalences} to this type of graphs. This study leads us to consider other classes of graphs and to study the relation existing between the valences of edge-magic and super edge-magic labelings and the well known problem of graph decompositions.

A decomposition of a simple graph $G$ is a collection $\{H_i: i \in [1,m] \}$ of subgraphs of $G$ such that $\cup_{i \in [1,m]}E(H_i)$ is a partition of the edge set of $G$. If the set $\{H_i: i \in [1,m] \}$ is a decomposition of $G$, then we denote it by $G \cong H_1 \oplus H_2 \oplus \dots \oplus H_m = \oplus_{i=1}^{m}H_i$.

We want to bring this introduction to its end by saying that the interested reader can also find excellent sources of information about the topic of graph labeling in \cite{BaMi,G,MhMi,SlMb,Wa}.

\section{More valences}
\label{section_morevalences}
As we have already mentioned in the introduction, not too much is known about the valences of (super) edge-magic labelings for the graph $C_m\odot \overline{K}_n$ when m is even. In fact,  as far as we know, the only papers that deal with (super) edge-magic labelings of $C_m\odot \overline{K}_n$ for $m$ even are \cite{FIM02,LMP1}. Hence almost all such results involve only odd cycles. Next, we study the edge-magic valences of $C_m\odot \overline{K}_n$ when $m$ is even. Unless otherwise specified, $\overrightarrow{G}$ denotes any orientation of $G$.
The next lemma is an generalization of Lemma 12 in \cite{LMP1}.
\begin{lemma}\label{repeatedvalences}
Let $g$ be a (super) edge-magic labeling of a graph $G$, and let $f_r$ be the super edge-magic labeling of $K_{1,n}^l$ that assigns label $r$ to the central vertex, $1 \leq r \leq n+1$. Then, 
\begin{itemize}
 \item[(i)] the induced (super) edge-magic labeling $\widehat{g}_r$ of $\overrightarrow{G} \otimes \overrightarrow{K}_{1,n}^l$ has valence $(n+1)(\hbox{val}(g)-2)+r+1$. \\
  \item[(ii)] Let $g'$ be a different (super) edge-magic labeling of $G$ with $\hbox{val}(g) < \hbox{val}(g')$, then $\hbox{val}(\widehat{g}_{n+1}) < \hbox{val}(\widehat{g}_{1}')$, where $\widehat{g}_r'$ is the induced (super) edge-magic labeling of $\overrightarrow{G} \otimes \overrightarrow{K}_{1,n}^l$ when $K_{1,n}^l$ is labeled with $f_r$ and $G$ with $g'.$
  \end{itemize}
\end{lemma}

\begin{proof}
The labeling $f_r$ of $\overrightarrow{K}_{1,n}^l$ has minimum induced sum $r+1.$ Thus, $\overrightarrow{K}_{1,n}^l \in \mathcal{S}_{n+1}^{r+1}.$ By Lemma \ref{valenceinducedproduct}, $\hbox{val}(\widehat{g}_{r})=(n+1)[\hbox{val}(g)-3]+r+1+n+1$, that is,
$\hbox{val}(\widehat{g}_{r})=(n+1)[\hbox{val}(g)-2]+r+1$. Let $g'$ be a different (super) edge-magic labeling of $G$ with $\hbox{val}(g) < \hbox{val}(g')$, then  $\hbox{val}(\widehat{g}_{n+1})= (n+1)[\hbox{val}(g)-2]+n+2 \leq (n+1)[\hbox{val}(g')-1-2]+n+2$. That is, $\hbox{val}(\widehat{g}_{n+1}) \leq (n+1) [\hbox{val}(g')-2]+1 <  \hbox{val}(\widehat{g}_{1}').$ Hence the result follows.
\end{proof}

\begin{theorem}\label{lowerboundvalences_1}
Let $G$ be an edge-magic $(p,q)$-graph. Then $ |\tau_{\overrightarrow{G} \otimes \overrightarrow{K}_{1,n}^l}| \geq (n+1)|\tau_{\overrightarrow{G}}|+2$. 
\end{theorem}

\begin{proof}
Let $f_r$ be the super edge-magic labeling of $K_{1,n}^l$ that assigns the label $r$ to the central vertex, $1\leq r \leq n+1$. Let $g:V(G) \cup E(G) \rightarrow [1,p+q]$ be an edge-magic labeling of  $G$. By Lemma \ref{repeatedvalences}, $\hbox{val}(\widehat{g}_r)=(n+1)[\hbox{val}(g)-2]+r+1$ and if $\hbox{val}(g) < \hbox{val}(g')$, then $\hbox{val}(\widehat{g}_{n+1}) < \hbox{val}(\widehat{g}_{1}')$ where $\widehat{g}_r$ is the induced edge-magic labeling of $\overrightarrow{G} \otimes \overrightarrow{K}_{1,n}^l$. Therefore, $ |\tau_{\overrightarrow{G} \otimes \overrightarrow{K}_{1,n}^l}| \geq (n+1)|\tau_{\overrightarrow{G}}|.$

Consider $\overrightarrow{K}_{1,n}^l \otimes \overrightarrow{G}.$ By Theorem \ref{producte_super_k}, $\hbox{val}(\tilde{g}_r)=(p+q)[n+r-1]+\hbox{val}(g), 1\leq r \leq n+1$ where $\tilde{g}_r$ is the induced labeling of $\overrightarrow{K}_{1,n}^l \otimes \overrightarrow{G}$ when $\overrightarrow{K}_{1,n}^l$ is labeled with $f_r$ and $\overrightarrow{G}$ with $g'.$
We claim that $\hbox{val}(\tilde{g}_1) < \hbox{val}(\hat{g}_1)$ and $\hbox{val}(\hat{g}_{n+1}) < \hbox{val}(\tilde{g}_{n+1})$. Assume to the contrary that $\hbox{val}(\tilde{g}_1) \geq \hbox{val}(\hat{g}_1)$, we get $\hbox{val}(g) \leq p+q+2$ which is a contradiction to Lemma \ref{maxminvalence}. Similarly, if $\hbox{val}(\hat{g}_{n+1}) \geq \hbox{val}(\tilde{g}_{n+1})$, we get $\hbox{val}(g) \geq 2(p+q)+1$ which again is a contradiction to Lemma \ref{maxminvalence}. Hence $|\tau_{\overrightarrow{G} \otimes \overrightarrow{K}_{1,n}^l}| \geq (n+1)|\tau_{\overrightarrow{G}}|+2.$
\end{proof}

By adding an extra condition on the smallest and the biggest valence, we can improve the lower bound given in the previous result.
\begin{theorem}\label{lowerboundvalences_2}
Let $G$ be an edge-magic $(p,q)$-graph. If $\alpha$ and $\beta$ are the smallest and the biggest valences of $G$, respectively, and $\beta-\alpha<(\alpha-(p+q+2))n$ then $ |\tau_{\overrightarrow{G} \otimes \overrightarrow{K}_{1,n}^l}| \geq (n+3)|\tau_{\overrightarrow{G}}|$.
\end{theorem}

\begin{proof}
The previous proof guarantees that, using Lemma \ref{repeatedvalences}, we get $ |\tau_{\overrightarrow{G} \otimes \overrightarrow{K}_{1,n}^l}| \geq (n+1)|\tau_{\overrightarrow{G}}|$. Next we will use Theorem \ref{producte_super_k} to complete the remaining valences. Consider now, the reverse order $\overrightarrow{K}_{1,n}^l \otimes \overrightarrow{G}.$ By Theorem \ref{producte_super_k}, $\hbox{val}(\widetilde{g}_r)=(p+q)[n+r-1]+\hbox{val}(g), 1\leq r \leq n+1$ where $\widetilde{g}_r$ is the induced labeling of $\overrightarrow{K}_{1,n}^l \otimes \overrightarrow{G}$ when $\overrightarrow{K}_{1,n}^l$ is labeled with $f_r$ and $\overrightarrow{G}$ with $g.$
Let $g$ be an edge-magic labeling of $G$ with valence $\alpha$ and $g'$ an edge-magic labeling with valence $\beta$. We claim that $\hbox{val}(\widetilde{g'}_1) < \hbox{val}(\widehat{g}_1)$ and $\hbox{val}(\widehat{g'}_{n+1}) < \hbox{val}(\widetilde{g}_{n+1})$. 

Assume to the contrary that $\hbox{val}(\widetilde{g'}_1) \ge \hbox{val}(\widehat{g}_1)$, then we get $\beta -\alpha\ge (\alpha-(p+q+2))n$ which is a contradiction to the statement. Similarly, if $\hbox{val}(\widehat{g'}_{n+1}) \ge \hbox{val}(\widetilde{g}_{n+1})$, we get $\beta-\alpha\ge (1+2(p+q)-\beta)n$. Notice that, since $\alpha $ and $\beta$ correspond to the valences of two complementary labelings of $G$, $\beta=3(p+q+1)-\alpha$ and this inequality is equivalent to $\beta -\alpha\ge (\alpha-(p+q+2))n$ which is again a contradiction. Since by construction of the induced labeling, if $\hbox{val}(g)< \hbox{val} (g')$, then $\hbox{val}(\tilde g_r)< \hbox{val} (\tilde g'_r)$, we obtain
$\hbox{val}(\tilde g_1)< \ldots < \hbox{val} (\tilde g'_1)<\hbox{val}(\hat g_1)<\ldots <\hbox{val}(\hat g'_{n+1})<\hbox{val}(\tilde g_{n+1})< \ldots < \hbox{val} (\tilde g'_{n+1}).$ Hence $|\tau_{\overrightarrow{G} \otimes \overrightarrow{K}_{1,n}^l}| \geq (n+3)|\tau_{\overrightarrow{G}}|.$
\end{proof}

\begin{figure}{b}
\begin{center}
\includegraphics[scale=0.90]{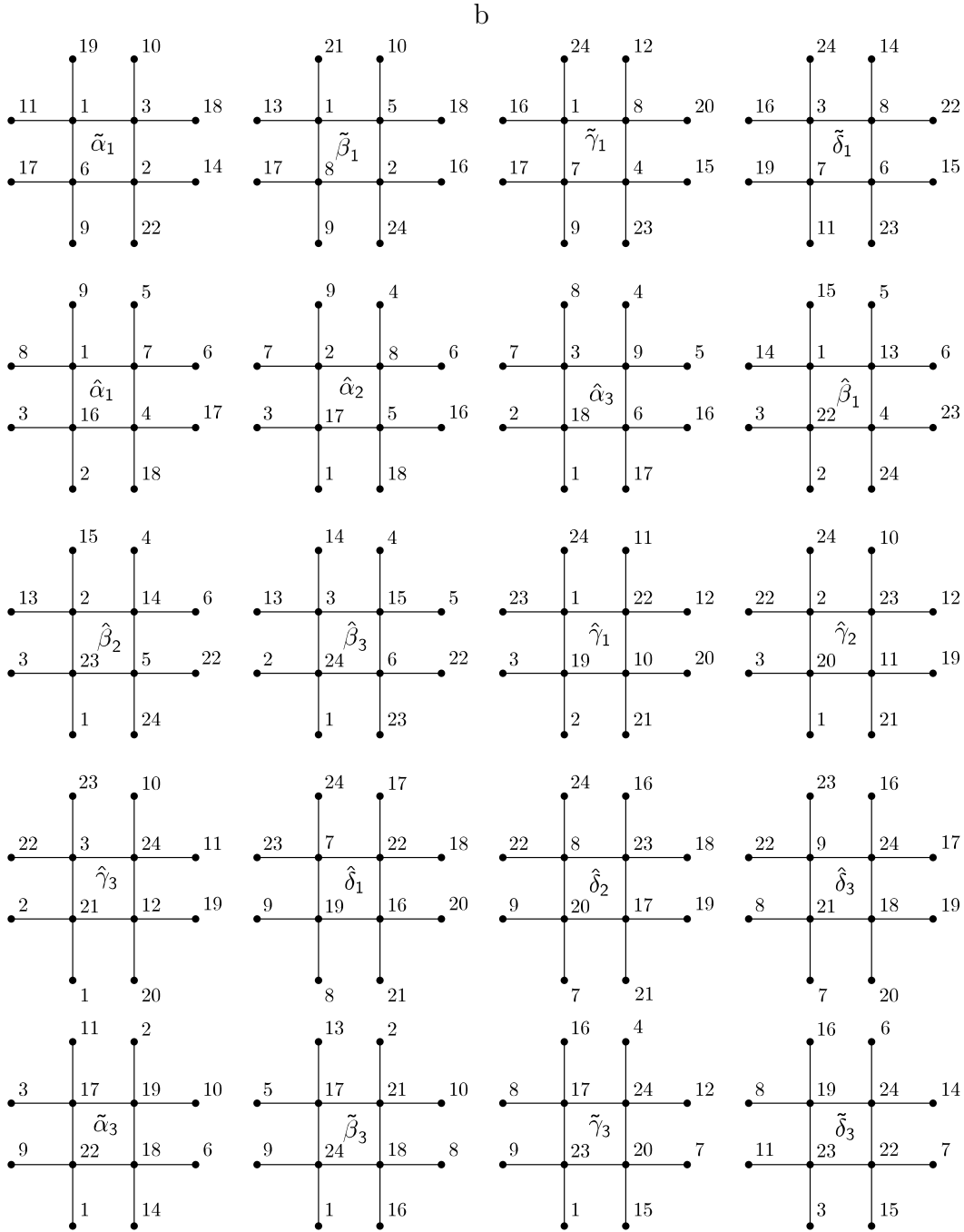}
\caption{All theoretical valences are realizable for $C_4 \odot \overline{K}_2.$}
\label{allvalences of c4coronak2}
\end{center}
\end{figure}

\begin{corollary}\label{bipartite_2_regular}
Let $G$ be any edge-magic (bipartite) 2-regular graph. Then $|\tau_{G\odot\overline{K}_n}| \geq (n+1)|\tau_G|+2$.
\end{corollary}

\begin{proof}
Let $G=C_{m_1}\oplus C_{m_2}\oplus \cdots \oplus C_{m_k}$ and let $\overrightarrow{G}=C_{m_1}^+\oplus C_{m_2}^+\oplus \cdots \oplus C_{m_k}^+$ be an orientation of $G$ in which each cycle is strongly oriented. Then $\overrightarrow{G}\otimes \overrightarrow{K}_{1,n}^l=(C_{m_1}^+\otimes \overrightarrow{K}_{1,n}^l) \oplus (C_{m_2}^+\otimes \overrightarrow{K}_{1,n}^l) \oplus \cdots \oplus (C_{m_k}^+\otimes \overrightarrow{K}_{1,n}^l)$. Note that since $G$ is bipartite, all cycles should be of even length and by definition of $\otimes$-product, $G \odot \overline{K}_{n} \cong und(\overrightarrow{G}\otimes \overrightarrow{K}_{1,n}^l)$. Thus by Theorem \ref{lowerboundvalences_1}, $|\tau_{G\odot\overline{K}_n}| \geq (n+1)|\tau_G|+2$.
\end{proof}

\begin{example}
Let $g$ be an edge-magic labeling of $\overrightarrow{C_4}$ and $f_r$ be the super edge-magic labeling of $\overrightarrow{K}_{1,2}^l$ that assigns the label $r$ to the central vertex, $1 \leq r \leq 3.$ Then the valence of the induced labeling $\widehat{g}_r$ is $\hbox{val}(\widehat{g}_r)=3(\hbox{val}(g)-2)+r+1 \in [3(\hbox{val}(g)-2)+2,3(\hbox{val}(g)-2)+4]$. Let 
$\alpha: 1\bar 56\bar 42\bar 73\bar81$, 
$\beta =1\bar 75\bar 62\bar 38\bar41$, 
$\gamma= 1\bar 58\bar 24\bar 37\bar61$ and 
$\delta= 8\bar 43\bar 57\bar 26\bar18$, where $i\bar mj$  indicates that $m$ is the label assineg to the edge $ij$. Since $\tau_{C_4}=[12,15]=[\hbox{val}(\alpha),\hbox{val}(\beta)]$ we get  different $12$ edge-magic valences $[32,43]$ for the induced labeling of $C_4 \odot \overline{K}_2 \cong und(\overrightarrow{C_4} \otimes \overrightarrow{K}_{1,2}).$ Moreover, since the condition $\hbox{val}(\beta)-\hbox{val}(\alpha)<(\hbox{val}(\alpha)-(p+q+2))n$, is satisfied for $n\ge 2$, by using Theorem \ref{producte_super_k}, $\hbox{val}(\widetilde{g}_r)=8(1+r)+\hbox{val}(g)$ which gives, associated to a labeling $g$ two new valences, namely $\hbox{val}(\widetilde{g_1})$ and $\hbox{val}(\widetilde{g_3})$ which gives in total $20$ valences. The induced labelings  and they are shown in Fig. \ref{allvalences of c4coronak2}, according to the notation introduced above (for clarity reasons, only the labels of the vertices are shown). Notice that, by using the missing labels, there is only one way to complete the edge-magic labelings obtained in Fig. \ref{allvalences of c4coronak2}. 
The minimum induced sum together with the maximum unused label provides the valence of the labeling.
\end{example}
\begin{remark}
For a given even $m$, the magic interval for crowns of the form $C_m\odot \overline{K}_n$ is $[mn+2+5m/2,2mn+ 1+7m/2]$ ( see Section 2, in \cite{PEM_LMR}). Thus, for $m=4$, the magic interval is $[28,47]$. Hence,  the crown $C_4\odot \overline{K}_2$ is perfect edge-magic. 

\end{remark}
It is well known that all cycles are edge-magic \cite{GodSla}. Thus, the following corollary follows:

\begin{corollary}
Fix $m \in \mathcal{N}$. Then $\lim_{n\to\infty} |\tau_{C_m\odot \overline{K}_n}|=\infty$.
\end{corollary}

A similar argument to that of the first part in Theorem \ref{lowerboundvalences_1} can be used to prove the following theorem.

\begin{theorem}\label{sem lowerbound}
Let $G$ be a super edge-magic graph. Then $ |\sigma_{\overrightarrow{G} \otimes \overrightarrow{K}_{1,n}^l}| \geq (n+1)|\sigma_{\overrightarrow{G}}|$.
\end{theorem}

\section{A relation between (super) edge-magic labelings and graph decompositions}
\label{section_decompositions}

Let $G$ be a bipartite graph with stable sets $X=\{x_i\}_{i=1}^s$ and $Y=\{y_j\}_{j=1}^t$. Assume that $G$ admits a decomposition $G\cong H_1\oplus H_2$. Then we denote by $S_2(G;H_1,H_2)$ the graph with vertex and edge sets defined as follows:
\begin{eqnarray*}
  V(S_2(G;H_1,H_2)) &=& X\cup Y\cup X'\cup Y', \\
  E(S_2(G;H_1,H_2)) &=& E(G)\cup \{x_iy_j': x_iy_j\in E(H_1)\}\cup \{x_i'y_j: x_iy_j\in E(H_2)\},
\end{eqnarray*}
where $X'=\{x_i'\}_{i=1}^s$ and $Y'=\{y_j'\}_{j=1}^t$.

We are ready to state and prove the next theorem.

\begin{theorem}\label{theo: SEM new_bipartite graph}
Let $G$ be a bipartite (super) edge-magic simple graph with stable sets $X$ and $Y$. Assume that $G$ admits a decomposition $G\cong H_1\oplus H_2$. Then, the graph $S_2(G;H_1,H_2)$ is (super) edge-magic.
\end{theorem}

\begin{proof}
Let $f$ be a (super) edge-magic labeling of $G$, and assume that the edges of $H_1$ are directed from $X$ to $Y$ and the edges of $H_2$ are directed from $Y$ to $X$ in $G$, obtaining the digraph $\overrightarrow{G}$. Let $\overrightarrow{K}_{1,1}^l$ be the super edge-magic labeled digraph with $V(\overrightarrow{K}_{1,1}^l)=\{1,2\}$ and $E(\overrightarrow{K}_{1,1}^l)=\{(1,1),(1,2)\}$. By Theorem \ref{spk}, we have that the graph und$(\overrightarrow{G}\otimes \overrightarrow{K}_{1,1}^l)$ is (super) edge-magic. Moreover, an easy check shows that the bijective function $\phi: V(\overrightarrow{G}\otimes \overrightarrow{K}_{1,1}^l)\rightarrow V(S_2(G;H_1,H_2))$ defined by $\phi (v,1)=v$ and $\phi (v,2)=v'$ is an isomorphism between und$(\overrightarrow{G}\otimes \overrightarrow{K}_{1,1}^l)$ and $S_2(G;H_1,H_2)$. Therefore, the graph $S_2(G;H_1,H_2)$ is (super) edge-magic.
\end{proof}

Next, we show an example.
\begin{example}
Consider the edge-magic labeling of $K_{3,3}$ shown in Fig. \ref{Fig_7}. The same figure shows a partition of the edges and  a possible orientation of them when $X=\{1,2,3\}$ and $Y=\{4,8,12\}$. The construction given in the proof of Theorem \ref{theo: SEM new_bipartite graph} when each vertex $(a,i)$ is labeled $2(a-1)+i$ and each edge $(a,i)(b,j)$ is labeled $2(e-1)+4-(i+j)$ (where $e$ is the label of $(a,b)$ in $D$) results into the graph in Fig. \ref{Fig_8}.

\begin{figure}[h]
  \centering
  \includegraphics[width=79pt]{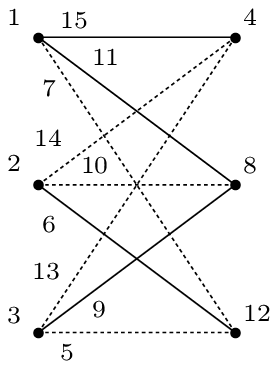}\includegraphics[width=79pt]{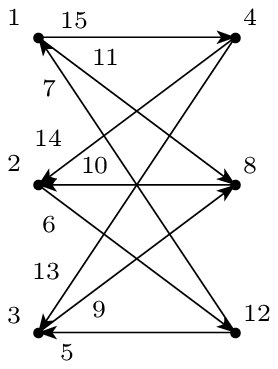}\\
  \caption{A decomposition of $K_{3,3}$ and the induced orientation.}\label{Fig_7}
\end{figure}
\begin{figure}[h]
  \centering
  \includegraphics[width=192pt]{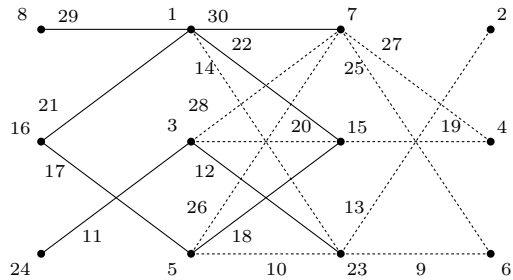}\\
  \caption{An edge-magic labeling of $S_2(K_{3,3};H_1,H_2)$.}\label{Fig_8}
\end{figure}
\end{example}

Kotzig and Rosa \cite{KotRos70} proved that every complete bipartite graph is edge-magic. It is clear that Theorem \ref{theo: SEM new_bipartite graph} works very nicely when the graph $G$ under consideration is a complete bipartite graph and many new edge-magic graphs can be obtained.  Theorem \ref{theo: SEM new_bipartite graph} can be easily extended. Let us do so next.

Let $G$ be a bipartite graph with stable sets $X=\{x_i\}_{i=1}^s$ and $Y=\{y_j\}_{j=1}^t$. Assume that $G$ admits a decomposition $G\cong H_1\oplus H_2$. Then we define $S_{2n}(G;H_1,H_2)$ to be the graph with vertex and edge sets as follows:
\begin{eqnarray*}
  V(S_{2n}(G;H_1,H_2)) &=& X\cup Y\cup (\cup_{k=1}^nX_k)\cup (\cup_{k=1}^nY_k), \\
   E(S_{2n}(G;H_1,H_2)) &=& E(G)\cup \{x_iy_j^k: x_iy_j\in E(H_1)\}\cup \{x_i^ky_j: x_iy_j\in E(H_2)\},
\end{eqnarray*}
where $X_k=\{x_i^k\}_{i=1}^s$ and $Y_k=\{y_j^k\}_{j=1}^t$.

\begin{lemma}\label{S2n isomorphism}
Let $G$ be a bipartite simple graph with stable sets $X$ and $Y$. Assume that $G$ admits a decomposition $G\cong H_1\oplus H_2$. Then, there exists an orientation of $G$ and $K_{1,n}^l$, namely $\overrightarrow{G}$ and $\overrightarrow{K}_{1,n}^l$ respectively, such that $S_{2n}(G;H_1,H_2) \cong und(\overrightarrow{G}\otimes \overrightarrow{K}_{1,n}^l).$
\end{lemma}
\begin{proof}
Assume that the digraph $\overrightarrow{G}$ is obtained from $G$ by orienting the edges of $H_1$ from $X$ to $Y$ and the edges of $H_2$ from $Y$ to $X$ in $G$. Let $\overrightarrow{K}_{1,n}^l$ be the digraph with $V(\overrightarrow{K}_{1,n}^l)=[1,n+1]$ and $E(\overrightarrow{K}_{1,n}^l)=\{(1,k):k \in [1,n+1]\}$. An easy check shows that the bijective function $\phi: V(\overrightarrow{G}\otimes \overrightarrow{K}_{1,n}^l)\rightarrow V(S_{2n}(G;H_1,H_2))$ defined by $\phi (v,1)=v$ and $\phi (v,k+1)=v^{k}, \ k \in [1,n]$ is an isomorphism between und$(\overrightarrow{G}\otimes \overrightarrow{K}_{1,n}^l)$ and $S_{2n}(G;H_1,H_2)$.
\end{proof}

We are ready to state and prove the next theorem.

\begin{theorem}\label{coro: SEM new_multipartite graph}
Let $G$ be a bipartite (super) edge-magic simple graph with stable sets $X$ and $Y$. Assume that $G$ admits a decomposition $G\cong H_1\oplus H_2$. Then, the graph $S_{2n}(G;H_1,H_2)$ is (super) edge-magic.
\end{theorem}
\begin{proof}
Let $f$ be a (super) edge-magic labeling of $G$, and assume that the edges of $H_1$ are directed from $X$ to $Y$ and the edges of $H_2$ are directed from $Y$ to $X$ in $G$, obtaining the digraph $\overrightarrow{G}$. Let $\overrightarrow{K}_{1,n}^l$ be the super edge-magic labeled digraph with $V(\overrightarrow{K}_{1,n}^l)=[1,n+1]$ and $E(\overrightarrow{K}_{1,n}^l)=\{(1,k):k \in [1,n+1]\}$. By Theorem \ref{spk}, we have that the graph und$(\overrightarrow{G}\otimes \overrightarrow{K}_{1,n}^l)$ is (super) edge-magic. By Lemma \ref{S2n isomorphism}, $S_{2n}(G;H_1,H_2) \cong und(\overrightarrow{G}\otimes \overrightarrow{K}_{1,n}^l).$ Therefore, the graph $S_{2n}(G;H_1,H_2)$ is (super) edge-magic.
\end{proof}

With the help of Lemma \ref{k1nsem}, we can generalize Theorem \ref{coro: SEM new_multipartite graph} very easily. We do it in the following two results.

\begin{theorem}\label{S_2n SEM}
Let $G$ be a bipartite super edge-magic simple graph with stable sets $X$ and $Y$. Assume that $G$ admits a decomposition $G\cong H_1\oplus H_2$. Then $|\sigma_{S_{2n}(G;H_1,H_2)}| \geq (n+1)|\sigma_G|$.
\end{theorem}
\begin{proof}
Let $h$ be a super edge-magic labeling of $G$, and assume that the edges of $H_1$ are directed from $X$ to $Y$ and the edges of $H_2$ are directed from $Y$ to $X$ in $G$, obtaining the digraph $\overrightarrow{G}$. Let $f_r$ be the super edge-magic labeling of $\overrightarrow{K}_{1,n}^l$ that assigns the label $r$ to the central vertex with $\hbox{val}(f_r)=2n+3+r, \ 1 \leq r \leq n+1$. Then by Lemma \ref{S2n isomorphism}, $S_{2n}(G;H_1,H_2) \cong und(\overrightarrow{G}\otimes \overrightarrow{K}_{1,n}^l)$ and by Theorem \ref{coro: SEM new_multipartite graph}, it is super edge-magic. By Theorem \ref{sem lowerbound}, $|\sigma_{S_{2n}(G;H_1,H_2)}| \geq (n+1)|\sigma_G|$.
\end{proof}

A similar argument to the one of Theorem \ref{S_2n SEM}, but now using Theorem \ref{lowerboundvalences_1}, allows us to prove the following theorem.

\begin{theorem} \label{S_2n EM}
Let $G$ be a bipartite edge-magic simple graph with stable sets $X$ and $Y$. Assume that $G$ admits a decomposition $G\cong H_1\oplus H_2$. Then $|\tau_{S_{2n}(G;H_1,H_2)}| \geq (n+1)|\tau_G|+2$.
\end{theorem}

Once again, we have the following two easy corollaries.

\begin{corollary}
Let $G$ be a bipartite super edge-magic simple graph with stable sets $X$ and $Y$. If $G$ admits a decomposition $G\cong H_1\oplus H_2$, then $\lim_{n\to\infty}|\sigma_{S_{2n}(G;H_1,H_2)}|=\infty$.
\end{corollary}

\begin{corollary}
Let $G$ be a bipartite edge-magic simple graph with stable sets $X$ and $Y$. If $G$ admits a decomposition $G\cong H_1\oplus H_2$, then $\lim_{n\to\infty}|\tau_{S_{2n}(G;H_1,H_2)}|=\infty$.
\end{corollary}

At this point, consider any graph $G^*$ whose vertex set admits a partition of the form $V(G^*)=X\cup Y\cup_{k=1}^n X_k \cup_{k=1}^n Y_k$ and that decomposes as a union of three bipartite graphs $G^*\cong G\oplus H_1\oplus H_2$, where $G^*[X\cup Y]\cong G$, $G^*[X\cup Y_k]\cong H_1$ and $G^*[X_k\cup Y]\cong H_2$ for all $k\in [1,n]$. By Theorem \ref{theo: SEM new_bipartite graph}, we have the following remarks.
\begin{remark}
If $G$ is a (super) edge-magic graph and $G^*$ is not, then $H_1$ and $H_2$ do not decompose $G$.
\end{remark}
\begin{remark}
If $|\sigma_{S_{2n}(G;H_1,H_2)}| < (n+1)|\sigma_G|$ provided that $G$ is a bipartite super edge-magic graph, then $G\not \cong H_1\oplus H_2 $.
\end{remark}
\begin{remark}
If $|\tau_{S_{2n}(G;H_1,H_2)}| < (n+1)|\tau_G|+2$ provided that $G$ is a bipartite edge-magic graph, then $G\not \cong H_1\oplus H_2 $.
\end{remark}
We will bring this section to its end, by mentioning that, although  some labelings involving differences as for instance, graceful labelings and $\alpha$-valuations have a strong relationship with graph decompositions, the results mentioned in this section are the only ones known relating the subject of decompositions with addition type labelings. This is why we consider these results interesting.

\section{Conclusions}
The goal of this paper is to show a new application of labeled super edge-magic (di)graphs to graph decompositions. The relation among labelings and decompositions of graphs is not new. In fact, one of the first motivations in order to study graph labelings was the relationship existing between graceful labelings of trees and decompositions of complete graphs into isomorphic trees. What we believe that it is new and surprising about the relation established in this paper is that, as far as we know, there are no relations between labelings involving sums and graph decompositions. In fact, we believe that this is the first relation found in this direction and we believe that to explore this relationship is a very interesting line for future research.

\vspace{0.5cm}

\noindent{\bf Acknowledgements} The research conducted in this document by the first author has been supported symbolically by the Catalan Research Council under grant 2014SGR1147.

\end{document}